\documentclass{amsart}

\usepackage[english]{my-shortcuts}
\usepackage{xcolor}
\usepackage{a4wide}
\usepackage{dsfont}
\usepackage{amsmath,amsfonts,amssymb,latexsym,amsthm}
\usepackage{comment} 
\usepackage{graphicx}
\usepackage{caption} 

\newcommand{\RR}{\mathbb{R}}

\newcommand{\CC}{\mathbb{C}}
\newcommand{\HH}{\mathbb{H}}

\newcommand{\re}{\mathfrak{R}}
\newcommand{\im}{\mathfrak{I}}

\begin{document}

\title{%Identities for 
Multiple integral formulas for weighted zeta moments: the case of the sixth moment}
\author{S\'ebastien Darses --- Joseph Najnudel}
\date{}

\begin{abstract}
    We prove exact formulas for weighted $2k$th moments of the Riemann zeta function for all integer $k\geq 1$ in terms of the analytic continuation of an auto-correlation function. This latter enjoys several functional equations. One of them, following from a fundamental lemma of Bettin and Conrey (2013), yields to new formulas for the moments: our second main result is the case $k=3$, but there is no obstruction to obtain higher moments. This generalizes results by Titchmarsh (1928) for $k=1$ and $k=2$. %The process allows to apply an amazing method of Titchmarsh (1936) that works for $k=2$.  One could figure out how an intriguing multiplicative structure scales up while comparing $k=1,2,3$.
    %These ones only involve the divisor function $d=d_2$, replicated $k$ times, but the price to pay is a non-trivial interaction with $k-1$ multiple integrals.
    A basic and powerful tool is a special Fourier transform unveiled by Ramanujan (1915).
    In a nutshell, the new idea is to consider the associated structures to $\Gamma^k\zeta^k$, which enjoy remarkable properties that are not satisfied by the more classically studied structure $\Gamma\zeta^k$.
\end{abstract}

\address{IRL CRM--CNRS, Universit\'e de Montr\'eal, Centre National de la Recherche Scientifique --\newline
Aix-Marseille Universit\'e, CNRS, Centrale Marseille, I2M, Marseille, France} 
\email{sebastien.darses@univ-amu.fr}

\address{University of Bristol, School of Mathematics, Bristol, United Kingdom}
\email{joseph.najnudel@bristol.ac.uk}


\maketitle

%\color{red} Comments: the formula in Theorem 1.2 is not exactly the same as the formula for $J$ end of p. 14 (different signs). To be corrected? Is it equal? The section Bettin-Conrey (end p.4, beginning p. 5) should be completed. I have completed the proof of Lemma 2.3 ($A = \psi + r$). I have changed the proof of Lemma 2.6 ($k$ fold convolution), in a setting where the induction on $k$ looks more obvious. I have added a few references for the max of zeta on random short intervals. Do we add more references on max on $[0,T]$? On the known bounds on moments? (e.g. 6th moment). 
%\color{black}

\section{Introduction}

\subsection{Main results}

The Riemann zeta function, the weighted one Eisenstein series, and the auto-correlation function $A$ introduced in \cite{DH21a}, are defined by:
\begin{eqnarray*}
\zeta(z) & = &\sum_{n\ge1} \frac{1}{n^z}, \quad \re(z)>1, \\
S_0(z) & = & \sum_{n\ge1}d(n)e(nz), \quad \im(z)>0, \\
A(z) & = & \int_0^\infty \left(\frac{1}{xz}-\frac{1}{e^{xz}-1}\right)\left(\frac{1}{x}-\frac{1}{e^{x}-1}\right)dx, \qquad \re(z)>0,
\end{eqnarray*}
where $\mathfrak{R}(z)$ (resp. $\mathfrak{I}(z)$) is the real part (resp. imaginary part) of a complex number $z$; 
%$d(n)=\tau(n)=\sigma_0(n)=d_2(n)$ 
$d(n)$ denotes the number of positive divisors of an integer $n$ and $e(z)=e^{2\pi iz}$, $i^2=-1$. We set $\R_{>0}=(0,\infty)$ and $\mathds{1}_{\cal E}$ is the indicator function of a set $\cal E$.
A general formula relates the following weighted zeta moments to the analytic continuation of $A$ on $\CC'=\CC\setminus (-\infty,0]$ (see Section \ref{analyticcontinuation}):
\begin{theo}\label{main-lemma}
For all integer $k\ge1$ and all $\delta\in(0,\pi)$,
\begin{eqnarray*}
    \int_{-\infty}^{\infty} 
 \left|\zeta\left(\frac{1}{2}+it\right)\right|^{2k} \frac{e^{ k(\pi -\delta) t}}{\cosh(\pi t)^k}\ dt
& = & 2\ \frac{e^{ik\frac{\delta-\pi}{2}}}{\pi^{k-1} }
 \int_{\RR_{>0}^{k}}  
\prod_{j=1}^{k} A(- u_j e^{i \delta})\ \delta_1\left( \prod_{j=1}^k u_j \right) \prod_{j = 1}^k du_j,
\end{eqnarray*}
where $\delta_1(u_1\cdots u_k)du_1\cdots du_k$ is the surface measure supported on $\{u_1\cdots u_k=1\}\subset\RR_{>0}^{k}$, which, seen as a distribution, is $\underset{\varepsilon \rightarrow 0^+}{\lim}\ \varepsilon^{-1} \mathds{1}_{1\leq u_1\cdots u_k \leq 1+\varepsilon}$. In particular, $\delta_1(u_1)du_1$ is the Dirac mass at $1$.
\end{theo}

To obtain more explicit formulas, we use three (exact) functional equations of $A$, a crucial one following from a fundamental lemma of Bettin and Conrey (see Section \ref{bettin-conrey}). As an application:

\begin{theo}\label{main-theorem}
For all $\delta\in(0,\pi/2)$,
\begin{eqnarray*}
\int_{-\infty}^{\infty}  \left|\zeta\left(\frac{1}{2}+it\right)\right|^{6} \frac{e^{ 3(\pi -\delta) t}}{\cosh(\pi t)^3}\ dt
    & = & 96\pi\  \re\int_1^\infty\int_1^{\infty} e^{i\frac{\delta}{2}} S_0(e^{i\delta} u) S_0(e^{i\delta} v) S_0(-e^{-i\delta} uv) dudv \ \\
        &  & \ -\frac{12}{\pi^2}\ \re \left(i e^{i\frac{\delta}{2}} \sum_{j=1}^5 R_j(\delta) \right),
\end{eqnarray*}
where the following remainders satisfy, for $\delta > 0$ sufficiently small, 
\begin{eqnarray*}
R_1(\delta) = 2\int_0^1\int_0^1 S(u)R(v)\b S(uv) dudv 
    & \ll & \delta^{-1} \log^{9/2}(\delta^{-1}) \\
R_2(\delta) = \int_0^1\int_0^1 S(u)S(v)\b R(uv)dudv     & \ll & \delta^{-1} \log^{4}(\delta^{-1}) \\
R_3(\delta) = \int_0^1\int_0^1 R(u)R(v)\b S(uv) dudv & \ll & \delta^{-1/2} \log^{5/2}(\delta^{-1}) \\
R_4(\delta) = 2\int_0^1\int_0^1 R(u)S(v)\b R(uv)dudv 
    & \ll & \delta^{-1/2} \log^{2}(\delta^{-1})\\
R_5(\delta) = \int_0^1\int_0^1 R(u)R(v)\b R(uv)dudv         & \ll  & 1,
\end{eqnarray*}
with, $\gamma$ being Euler's constant, and $u>0$,
\begin{eqnarray*} 
S(u) & = & 2\pi i\ e^{-i\delta} S_0(-e^{-i\delta}/u)/u \\
R(u) & = & -A(ue^{i\delta}) -\log(u) + \log(2\pi) -\gamma + \frac{i\pi}{2} - i\delta. 
\end{eqnarray*}
\end{theo}
In this article, the notation $X \ll Y$ or $X=O(Y)$ (resp. $X \ll_\alpha Y$ or $X=O_\alpha(Y)$) means $|X|\leq C Y$ for some constant $C>0$ (resp. $|X |\leq  C_{\alpha} Y$ for some $C_{\alpha} >0$ depending only on $\alpha$). Sometimes, for the sake of readability, the dependence on $\delta$ is removed, like within $S$ and $R$.
\medskip

The previous formula generalizes the ones that can be obtained for the 2nd and 4th moments:

\begin{theo}[Titchmarsh 1928, reformulated] \label{moment4}
For all $\delta\in(0,\pi/2)$,
\begin{eqnarray*}
\int_{-\infty}^{\infty} \left|\zeta\left(\frac{1}{2}+it\right)\right|^{2} \frac{e^{(\pi-\delta) t}}{\cosh(\pi t)} dt
     & = & 4\pi\ e^{i\frac{\delta}{2}} S_0(e^{i\delta})
     + \ 2i e^{-i\frac{\delta}{2}}\left( \log(2\pi)-\gamma -\frac{i\pi}{2} - A(e^{-i\delta}) + i\delta \right),
\end{eqnarray*}
and 
\begin{eqnarray*}
\int_{-\infty}^{\infty}  \left|\zeta\left(\frac{1}{2}+it\right)\right|^{4} \frac{e^{ 2(\pi -\delta) t}}{\cosh(\pi t)^2} dt
    & = & 16\pi \int_1^\infty \left|S_0(e^{i\delta}u)\right|^2 du  + \wdt{R_1}(\delta) + \wdt{R_2}(\delta),
\end{eqnarray*}
where, $R$ and $S$ being defined as above, the remainders read and satisfy, for $\delta > 0$ small enough, 
\begin{eqnarray*}
\wdt{R_1}(\delta) = \frac{8}{\pi}\ \re\int_0^1 \b S(u) R(u) du \ \ll\  \delta^{-1/2} \log^{2}(\delta^{-1}),  &  \quad & 
\wdt{R_2}(\delta) = \frac{4}{\pi}\int_0^1 |R(u)|^2 du \ll 1.
\end{eqnarray*}
\end{theo}

At this stage, let us make four comments. 

(i) One interest of Theorem \ref{main-lemma} is that it provides a general framework in which the formulas of Titchmarsh in Theorem \ref{moment4}
for $k=1$ and $k=2$ can be interpreted (as an example, the $S_0(e^{i\delta}\cdot 1)$ in the 2nd moment stems from a Dirac mass at $u=1$).
%Titchmarsh \cite{Tit28} used a method based on Fourier analysis to %study the asymptotics of the 2nd and 4th weighted moments as obtain the main terms in Theorem \ref{moment4}. Our method relies on the functions $A$ and $B$ (defined in Section \ref{analyticcontinuation}), their analytic continuation and their Mellin or Fourier transforms, while the core function in \cite{Tit28} is an {\em hybrid} function in between $A$ and $B$. 
As Titchmarsh \cite{Tit28}, our method is based on Fourier analysis, but it relies on two functions, $A$ and $B$ (defined in Section \ref{analyticcontinuation}), their analytic continuation and their Mellin or Fourier transforms, while the core function in \cite{Tit28} is an {\em hybrid} function in between $A$ and $B$. This allowed us to prove Theorem \ref{main-theorem} (case $k=3$).

(ii) Concerning higher moments, %for the more classical weight $e^{-\delta t}$, 
there exists a formula involving a $L^2$ norm on $(0,\infty)$ related to the inverse Mellin transform of $\Gamma\zeta^k$ (see \cite[p.160-161]{Tit86}). One interest of Theorem \ref{main-theorem} lies in the remarkable properties of $A$ and $B$ (related to Bettin and Conrey's period function $\psi$, see Section \ref{bettin-conrey}), which proved fruitful in locating the integral on $(1,\infty)$, and then {\em generalizing the main term} of the 4th moment. The threshold $1$ will turn out to be the same for all $k\geq 2$. We also obtain explicit remainders for which a $L^2$ norm on $(0,1)$ related to $A$ is controlled. %For instance, the remainder after $S_0$ in the second moment tends to $2i \left( \log(2\pi)-\gamma - i\pi/2 - A(1) \right)$ as $\delta\to 0^+$.

(iii) It is known that $S_0(e^{i\delta})$ and $\int_1^\infty |S_0(e^{i\delta}u)|^2 du $, for the 2nd and 4th moments respectively, gives the desired order of magnitude, namely $\delta^{-1}\log\left(\delta^{-1}\right)$ for the 2nd moment and $\delta^{-1}\log^4\delta^{-1}$ for the 4th one. The estimates of the remainders in our formula for the 6th moment are likely not optimal, but they are sufficient to show that the double integral main term is equivalent to the 6th moment as $\delta\to 0^+$, since it is known unconditionally that the 6th moment is at least of order $\delta^{-1}\log^9\delta^{-1}$.
The study of our main term will be performed in a forthcoming paper.

(iv) The weighted moments studied in this article are directly connected to the classic ones with exponential weights.  Indeed, 
\begin{eqnarray*}
    \int_{-\infty}^{\infty} 
 \left|\zeta\left(\frac{1}{2}+it\right)\right|^{2k} \frac{e^{ k(\pi -\delta) t}}{\cosh(\pi t)^k}dt  & = 
&  \int_{-\infty}^{0 } 
 \left|\zeta\left(\frac{1}{2}+it\right)\right|^{2k} \frac{e^{ k(\pi -\delta) t}}{\cosh(\pi t)^k} dt\\
   &  & \  +\ 2^k \int_{0}^{\infty} 
 \left|\zeta\left(\frac{1}{2}+it\right)\right|^{2k} \frac{e^{  -\delta k t}}{ (1 + e^{-2 \pi t})^k} dt
\\ & =  &  O_k(1) + 2^k \int_{0}^{\infty} 
e^{  -\delta k t} \left|\zeta\left(\frac{1}{2}+it\right)\right|^{2k} dt.
%\\ & = &  O_k(1) +(\delta k) 2^k  \int_{0}^{\infty} 
%e^{  -\delta k T} \left( \int_0^T  \left|\zeta\left(\frac{1}{2}+it\right)\right|^{2k}  dt \right) dT,
\end{eqnarray*}
These moments are also are connected with the moments of the form $\int_0^T |\zeta(1/2+it)|^{2k}dt$, via Tauberian theorems (see \cite[7.12]{Tit86}). 
The Lindel\"of hypothesis is equivalent to the fact that
\begin{eqnarray*}
\int_{-\infty}^{\infty} 
\left|\zeta\left(\frac{1}{2}+it\right)\right|^{2k} \frac{e^{ k(\pi -\delta) t} }{\cosh(\pi t)^k} dt
     & = & \delta^{-1 + o(1)}
\end{eqnarray*}
when $\delta \rightarrow 0^+$, for fixed $k$, and the moment conjectures by Keating and Snaith can be restated as 
 $$\int_{-\infty}^{\infty}  \left|\zeta\left(\frac{1}{2}+it\right)\right|^{2k} \frac{e^{ k(\pi -\delta) t} }{\cosh(\pi t)^k} dt \underset{\delta \rightarrow 0+}{\sim}  c_k\ \delta^{-1} \log^{k^2} (\delta^{-1})$$
where $c_k > 0$ is a suitable constant.

\medskip

\subsection{Previous works and motivation}
\subsubsection{The moments of $\zeta$}
The moments of the Riemann zeta function have been intensively studied in the last decades, starting from the work of 
Hardy and Littlewood \cite{HL18}, who proved the following equivalence 
$$
\frac{1}{T} \int_0^{T} \left|\zeta \left(\frac12 + it\right)\right|^{2} dt 
\underset{T \rightarrow \infty}{\sim} \log T.
$$
Ingham \cite{Ing26} then proved that
$$
\frac{1}{T} \int_0^{T} \left|\zeta \left(\frac12 + it\right)\right|^{4} dt 
\underset{T \rightarrow \infty}{\sim} (2 \pi^2)^{-1} \log^4 T,
$$
and it is conjectured  that the moment of order $2k$ is equivalent to a constant times 
$\log^{k^2} T$ for all $k \geq 0$. An explicit expression of the constant has been obtained by Keating and Snaith \cite{KS00}, using a conjectured connection between the Riemann zeta function and the characteristic polynomial of random unitary matrices (see also Katz and Sarnak \cite{KS99}). 
Conditionally on the Riemann hypothesis (RH), Soundararajan \cite{Sou09} proved that the $2k$th moment is dominated by $(\log T)^{k^2 + \varepsilon}$ for all $k \geq 0$, $\varepsilon > 0$, and Harper \cite{Har13} improved this estimate 
to $O_k((\log T)^{k^2})$. 

Unconditionally, only upper bounds with positive powers of $T$ depending on $k$ are known (see e.g. \cite[p.212]{Ivi03} with the notation $A=2k$, and \cite[p.173]{Tit86} in the setting of weighted moments). Again unconditionally, the lower bounds of the form $(\log T)^{k^2}$ are obtained for integers $k$ by Titchmarsh \cite[p.174]{Tit86}, for half integers by Ramachandra \cite{Ram78}, generalized by Heath-Brown \cite{HB81} to any rational, by Radziwi\l \l\ and Soundararajan \cite{RS13} to any real numbers $k>1$, and for $0<k\leq 1$ by Heap and Soundararajan \cite{HS22}.

Asymptotic expansions can be obtained in various settings for the 2nd and the 4th moment: either $\int_0^{T} |\zeta (\frac12 + it)|^{2} dt$ or weighted moments $\int_0^\infty|\zeta(\frac12 + it)|^2 e^{-\delta t}dt$, $\delta>0$ (see e.g. \cite[p.164]{Tit86}, and Bettin and Conrey \cite{BC13a} for convergent asymptotic series), and either $\int_0^{T} |\zeta (\frac12 + it)|^{4} dt$ or $\int_0^\infty|\zeta(\frac12 + it)|^4 e^{-\delta t}dt$ (see Heath-Brown \cite{HB79} and Atkinson \cite{Atk41}).

The study of moments also appears in a wide variety of problems: mean square of the product of the
Riemann zeta-function and a Dirichlet polynomial \cite{BCH85}, shifted moments and correlation of divisor sums (see \cite{CFKRS05,CK19,DFI94,Ng21}), spectral theory and trace formulas (see e.g. \cite{IM06}), exponential moments of the argument of  zeta \cite{Naj20} for instance. The moments of the Riemann zeta function are also related to extreme values on random intervals of the critical line: see for example the conjecture by Fyodorov, Hiary and Keating \cite{FHK12}, which has been partially solved by Arguin, Bourgade and Radziwi\l \l \,  in \cite{ABR20} and in \cite{ABR23}, after earlier progress by Najnudel in \cite{Naj18}, and by Arguin, Belius, Bourgade, Radziwi\l \l \, and Soundararajan in \cite{ABBRS19}. For the overall maximum of $\zeta$ on the vertical segment $[1/2, 1/2 + i T]$, the Lindel\"of hypothesis states that it grows slower than any positive power of $T$ when $T$ goes to infinity: Bourgain \cite{Bou17} proved that it is dominated by $T^{\alpha}$
for all $\alpha > 13/84$.

%The Lindel\"of hypothesis (LH), which is implied by RH, states that $|\zeta(1/2 + it)| \leq t^{o(1)} $ when $t > 1$ tends to infinity. From elementary bounds on the derivative of $\zeta$, one deduces that LH is equivalent to the fact that for all $k \geq 1$, $$\frac{1}{T} \int_0^{T} |\zeta (1/2 + it)|^{2k} dt = T^{o(1)}$$ when $T \rightarrow \infty$ and $k$ is fixed. 

\subsubsection{Bettin and Conrey's period function and its applications}

In a series of papers in 2013, Bettin and Conrey made several contributions to the theory of period functions related to Lewis-Zagier's theory of Maass forms, establishing reciprocity formulas, revealing remarkable property of cotangent sums, see \cite{BC13a,BC13b, MR16,ABB17}. Especially, they study the period function $\psi$ associated to a weight one Eisenstein series $E_1$. 
%which "measures" somehow the default of the modularity property satisfied by the Eisenstein series $E_{2k}$, $k\ge 2$. 
When comparing \cite[Theorem 1]{BC13b} and \cite[Theorem 1]{DH21a}, a connection between $\psi$ and $A$ had been found in \cite{DH21a}. This turned out to be contained in the proof of \cite[Lemma 1, Part one, p.5717]{BC13b} if one notices that the inverse Mellin transform of $\zeta(s)\zeta(1-s)/\sin(\pi s)$ is basically $A$ (see Section \ref{bettin-conrey}).

One application of their study of $\psi$, namely its expansion around $1$, is the following closed-form formula obtained in \cite{DH23}. For all integer $N\geq 0$, with $\sum_\varnothing :=0$,
\bean 
\frac{(-4)^N}{2}\int_{-\infty}^\infty t^{2N}  \left|\zeta\left(\frac{1}{2}+it\right)\right|^2 \frac{dt}{\cosh(\pi t)}  = \log(2\pi)-\gamma -4N + \left(\frac{4^N}{2}-1\right) B_{2N} + \sum_{j=2}^{2N} T_{2N,j}\frac{\zeta(j)B_{j}}{j}, \label{moment-poly}
\eean
where
$
T_{N,j} = (j-1)!\sum_{2\leq n \leq N} \binom{N}{n} 2^{n} \left[(-1)^n S(n+1,j) + (-1)^j S(n,j-1)\right]$ are integers, $B_j$ are the Bernoulli numbers and $S(n,j)$ the Stirling numbers of the second kind. As a determinate moment problem (in measure theory), this identity is a full characterisation of $\left|\zeta\left(\frac{1}{2}+it\right)\right|$. %The more the investigations progress, the more Bettin \& Conrey's lemma appears as a fundamental tool, motivating us to explore other applications.
This motivated us to explore further applications of the fundamental works of Bettin and Conrey.

The previous identities have been obtained while studying generalizations \cite{DH21b,ADH22} of the Nyman-Beurling criterion for RH (see \cite{BDBLS00}). One framework generates a Hankel Gram matrix $H$ whose diagonal entries are given by the above formula, containing the arithmetic of $\zeta$. RH then reads as an inverse problem $Hx=b$ where $b$ is a suitable target vector. 

%In short, changing the initial framework of "dilated" approximating structures allowed for reaching out Hankel Gram matrices $H$ whose diagonal entries are given by the above formula. RH then reads as an inverse problem $Hx=b$ where $b$ is a suitable target vector. 

\subsection{Additional notations, basic functions and formulas} Throughout the paper, $x,y,t,u,v$ are real numbers, and $s,z,w$ are complex numbers. 

Moreover, $j,k,\ell, m,n, q, r$ denote positive integers, while $i^2=-1$.

The Poincar\'e upper half-plane is $\HH=\{x+iy, y>0\}$, and we will also need to consider the right half-plane $\re_{>0}=\{x+iy, x>0\}$. 

Often we write $\Gamma\cdot \zeta (s)=\Gamma(s)\zeta(s)$, where the Gamma function is defined as usual as:
\bean
\Gamma(s) & = & \int_0^\infty e^{-x}x^{s-1}dx, \quad \mathfrak{R}(s)>0.
\eean
We recall Euler's reflection formula for $0<\re(s)<1$, and real $t$:
\bean
\Gamma(s)\Gamma(1-s) & = & \frac{\pi}{\sin(\pi s)} \\
\left|\Gamma\left(\frac{1}{2}+it\right)\right|^2 & = & \frac{\pi}{\cosh(\pi t)},
\eean
where $\d \cosh(x)=\frac{e^{x}+e^{-x}}{2}$. Set for $\re(z)> 0$,
\bea
\phi_1(z) & = & \frac{1}{e^{z}-1}-\frac{1}{z},
\eea
and $\phi_1(0)=-1/2$ (which is the Bernoulli number $B_1$). Notice that $\phi_1$ is continuous on $[0,\infty)$ and $\phi_1(x)<0$ for all $x\ge 0$, so that $A(u)>0$ for all $u>0$. \\

%\subsection{Fourier and Mellin transforms}

Let $f\in L^1(\RR)$. Its Fourier transform $\cal F f$ is well-defined as, for real $w$:
\begin{eqnarray*}
    \cal F f(w) & = & \int_{-\infty}^\infty e^{-ixw}f(x)dx,
\end{eqnarray*}
and also defined if $f\in L^2(\RR)$ by means of Fourier-Plancherel theory.
The inversion formula then reads for $g\in L^1(\RR)+L^2(\RR)$, and all real $x$,
\begin{eqnarray*}
    \cal F^{-1}g(x) & = & \frac{1}{2\pi} \int_{-\infty}^\infty e^{ixw}g(w)dw.
\end{eqnarray*}

Set $\re(s)=\sg$. For a function $g:\R^+\to\R$ such that $\d \int_0^\infty |g(x)|x^{\sg-1}dx<\infty$ for all $\sg\in(0,1)$, one can define its Mellin transform in the critical strip $\{0<\sigma<1\}$:
\bean
\cal Mg(s) & = & \int_0^\infty g(x)x^{s-1}dx.
\eean
%We recall Mellin-Plancherel isometry for $g\in L^2(0,\infty)$:
%\bean
%\int_0^\infty |g(x)|^2dx & =  & \frac{1}{2\pi}\int_{-\infty}^{\infty} \left|\widehat{g}\left(\frac{1}{2}+it\right)\right|^2dt.
%\eean

Let us recall the fundamental formula (see \cite{Tit86}, p. 23, formula (2.7.1)): 
\bea \label{mellin-gamma-zeta}
\cal M(\phi_1)(s) = \int_0^\infty \left(\frac{1}{e^{x}-1}-\frac{1}{x}\right)x^{s-1}dx = \Gamma\cdot\zeta(s), \quad 0<\mathfrak{R}(s)<1.
\eea

\subsection{Outline}
In Section \ref{AandB}, we provide the analytic continuation of the functions $A$ and $B : v \mapsto e^{v/2} A(e^v)$, and of their $k$-fold convolution for integers $k \geq 2$. Using Fourier inversion, we deduce a proof of Theorem \ref{main-lemma}. Moreover, from a reformulation of a result by Bettin and Conrey, we give an expression of $A$ in terms of Eisenstein series. 

In Section \ref{proofweightedmoment}, we apply results of Section \ref{AandB} in order to prove Theorem \ref{main-theorem} and Theorem \ref{moment4}. 

\medskip

\section{The functions $A$ and $B$, and their $k$-fold convolution} \label{AandB}

\subsection{Useful transforms and analytic continuation of $A$ and $B$} \label{analyticcontinuation}

At the end of his paper \cite{Ram15}, Ramanujan obtained a remarkable Fourier transform, involving the following function:
\bean
\R\to\R_{>0} , \quad v\longmapsto \int_0^\infty\left(\frac{1}{e^{xe^v}-1}-\frac{1}{xe^v}\right)\left(\frac{1}{e^{xe^{-v}}-1}-\frac{1}{xe^{-v}}\right)dx & =: & B(2v).
\eean
%See \cite{D11,Kim16} for many comments and recent proofs regarding formulas in \cite{Ram15}.
Due to the change of variable $x'=xe^{-v}$ in the integral, we have for all real $v$,
\begin{eqnarray} 
    B(v) & = & e^{v/2} A(e^v),
\end{eqnarray}
where $A$ is defined on real numbers $u>0$ by
\begin{eqnarray} \label{defA}
A(u) = \int_0^\infty \left(\frac{1}{xu}-\frac{1}{e^{xu}-1}\right)\left(\frac{1}{x}-\frac{1}{e^{x}-1}\right)dx = \int_0^\infty \phi_1(xu)\phi_1(x) dx.
\end{eqnarray}
We introduce as in \cite{BC13b} (p. 5716)  the following function $Q$ (which only differs here by a factor $\pi$, the authors considering $1/\sin(\pi s)$ instead of $\Gamma(s) \Gamma(1-s)$):
\begin{eqnarray*}
    Q(s) & = & \Gamma\cdot\zeta(s)\ \Gamma\cdot\zeta(1-s).
\end{eqnarray*}
Recall the basic facts $\zeta(\b s)=\b{\zeta(s)}$ and $\Gamma(\b s)=\b{\Gamma(s)}$, useful here when $s=1/2+it$.\\

We now give below two useful transforms of $A$ and $B$. The first identity (\ref{mellinA}) concerning $\cal M A$ is a remark that allows an important reformulation of Bettin and Conrey's Lemma 1 in \cite{BC13b}. The second one (\ref{fourierB}) concerning $\cal F B$ is a reformulation of an identity of Ramanujan \cite{Ram15} written in terms of $\cal F^{-1}$. In \cite{DH23}, we recalled a proof based on the Mellin-Plancherel isometry, and we provide here a short proof of (\ref{fourierB}) based on (\ref{mellinA}).
\begin{lemm} 
We have
\bea \label{mellinA}
\cal M A(s) & = & Q(s), \quad 0<\mathfrak{R}(s)<1.
\eea
Moreover, $B\in L^1(\R)$ and for all real $t$,
\bea \label{fourierB}
\cal F B(t) & = & \left|\zeta\left(\frac{1}{2}+it\right)\right|^{2} \frac{\pi}{\cosh(\pi t)}.
\eea
Finally, $B$ extends to a holomorphic function on $\cal W=\{x+iy, -\pi<y<\pi\}$, and $A$ extends to a holomorphic function on  $\C'$.
\end{lemm}

\begin{proof}
By the Fubini--Tonelli theorem, we have for $\sigma\in(0,1)$,
\begin{eqnarray}\label{cond-mellin-A}
\nonumber \int_0^\infty x^{\sigma-1}|A(x)|dx & = & \int_0^\infty \int_0^\infty \left( \frac{1}{t}-\frac{1}{e^t-1} \right)\left(\frac{1}{tx}-\frac{1}{e^{tx}-1}\right)x^{\sigma-1}dx\ dt  \\
    & = & \int_0^\infty \left(\frac{1}{x}-\frac{1}{e^{x}-1}\right)x^{\sigma-1}dx\int_0^\infty \left( \frac{1}{t} - \frac{1}{e^t-1} \right) t^{-\sigma} dt <\infty.
\end{eqnarray}
Hence we can write for $0<\mathfrak{R}(s)<1$, using (\ref{mellin-gamma-zeta}) and noting $-s=1-s-1$,
\bean
\cal M A(s) = \int_0^\infty \cal M(\phi_1(t\ \cdot))(s)\ \phi_1(t)dt = \cal M(\phi_1)(s) \int_0^\infty \phi_1(t) t^{-s}dt = Q(s).
\eean
We have $B\in L^1(\R)$ since, due to (\ref{cond-mellin-A}) with $\sigma=1/2$,
\bean
\int_{-\infty}^{\infty} 
|B(x)| dx = \int_{-\infty}^{\infty} 
e^{x/2} A(e^x) dx = \int_0^{\infty} y^{1/2-1} A(y) dy <\infty.
\eean
Hence we can compute
\begin{eqnarray*}
  \cal F B(t) = \int_{-\infty}^{\infty} e^{-i t x} 
e^{x/2} A(e^x) dx =  \int_0^{\infty} y^{1/2-it} A(y) \frac{dy}{y} = \cal M A\left(\frac{1}{2} - it\right).
\end{eqnarray*}
Therefore, from (\ref{mellinA}),
\begin{eqnarray*}
  \cal F B(t) = Q\left(\frac{1}{2} - it\right) = \left|\Gamma\cdot\zeta\left(\frac{1}{2}+it\right)\right|^2,
\end{eqnarray*}
which yields to the desired equality due to Euler's reflection formula.

Hence, we can deduce from (\ref{fourierB}) that for all real $z$,
\begin{eqnarray} \label{inversefourierB}
B(z) & = &  \frac{1}{2 \pi} 
\int_{-\infty}^{\infty} e^{i z t} Q(1/2 - it) dt =  \frac{1}{2 \pi} 
\int_{-\infty}^{\infty} e^{i z t}\left|\zeta\left(\frac{1}{2}+it\right)\right|^{2} \frac{\pi dt}{\cosh(\pi t)},
\end{eqnarray}
which is Ramanujan's formula written with $\zeta$ and $\cosh$ (instead of $\Xi$ and $\Gamma$).

Let then notice that the right hand side is analytic on $\cal W=\{x+iy, -\pi<y<\pi\}$ and provides the analytic continuation of $B$ on $\cal W$.
Writing 
\begin{eqnarray*}
    A(z) =  z^{-1/2} B(\log z) = e^{-(1/2) \log z} B(\log z),
\end{eqnarray*}
provides the analytic continuation of $A$ on $\C'$, through the principal branch of $\log : \C'\to \cal W$. 
\end{proof}

Notice that we have, for $z^{-1/2 + it}= e^{(-1/2 + it) \log z}$, 
\begin{equation} A(z) = \frac{1}{2 \pi} 
\int_{-\infty}^{\infty} z^{-1/2 + it}  Q(1/2 - i t) dt 
= \frac{1}{2 \pi} 
\int_{-\infty}^{\infty} z^{-1/2 + it} \left|\zeta\left(\frac{1}{2}+it\right)\right|^{2} \frac{\pi dt}{\cosh(\pi t)} .  \label{inverseMellinA}
\end{equation}
%\subsection{First estimates for $A$}

We end this subsection with two remarks:
\begin{enumerate}
    \item[(a)] Due to crude estimates on $\zeta(1/2+it)$ and its derivatives (from Cauchy's integral formula, see \cite{Ten22}), we deduce that $\cal F B$ belongs to the Schwarz class, and then $B$ does. We will need a "uniform" version of this result in the proof of Theorem \ref{main-lemma} (see Lemma \ref{boundB}).
\item[(b)] Finally, as claimed in the first page, the analytic extension of $A$ to $\re_{>0}$ can be simply performed through its initial definition on the real numbers, and the same remark can be done for $B$ on $\mathcal{W}$. 
\end{enumerate}

\begin{lemm}
    For all $z\in \re_{>0}$,
\begin{eqnarray} \label{Adroite}
A(z) & = & \int_0^\infty \left(\frac{1}{xz}-\frac{1}{e^{xz}-1}\right)\left(\frac{1}{x}-\frac{1}{e^{x}-1}\right)dx,
\end{eqnarray}
and for all $z \in \mathcal{W}=\{x+iy, -\pi<y<\pi\}$,  
\begin{eqnarray} \label{Adroite}
B(z) & = & \int_0^\infty\left(\frac{1}{e^{xe^{z/2}}-1}-\frac{1}{xe^{z/2}}\right)\left(\frac{1}{e^{xe^{-z/2}}-1}-\frac{1}{xe^{-z/2}}\right)dx .
\end{eqnarray}
\end{lemm}

\begin{proof}
We have, for $|z| \leq 1$,
\begin{eqnarray}
\left|\frac{1}{e^z - 1} - \frac{1}{z} \right| & = & O(1)
\end{eqnarray}
by the Taylor series expansion, which has a radius of convergence $2 \pi >1$. \\
On the other hand,  for $|z| > 1$, and 
$|\operatorname{Arg} (z)| < \alpha$ for some $\alpha \in (0, \pi/2)$, 
$$|e^{z} - 1| \geq e^{\Re(z)} - 1 \geq e^{ |z| \cos \alpha} - 1
  \gg_{\alpha} |z|,$$ 
and therefore 
\begin{eqnarray}
\left|\frac{1}{e^z - 1} - \frac{1}{z} \right| & = & O_{\alpha}(1/|z|).
\end{eqnarray}

Hence, for all $z\in \re_{>0}$ such that $|\operatorname{Arg} (z)| < \alpha$, 
\bean
\left|\frac{1}{e^z - 1} - \frac{1}{z} \right| & \ll_{\alpha} & \min( 1, |z|^{-1} ).
\eean
Now, for all $z_1, z_2 \in \re_{>0}$ such that $|\operatorname{Arg} (z_1)|, |\operatorname{Arg} (z_2)|  < \alpha$, 
\begin{eqnarray*}
\left| \left(\frac{1}{xz_1}-\frac{1}{e^{xz_1}-1}\right)\left(\frac{1}{xz_2}-\frac{1}{e^{xz_2}-1}\right) \right| & \ll_{\alpha} & \min( 1, |z_1|^{-1}|x|^{-1} )\min( 1, |z_2|^{-1}|x|^{-1} ).
\end{eqnarray*}
We deduce that in the lemma, the integral corresponding to $A$ is uniformly convergent for $z$ in any compact subset of $ \re_{>0}$, and the integral corresponding to $B$ is uniformly convergent for $z$ in any compact subset of $\mathcal{W}$. 
Hence, if we compose these integrals with a further complex integration along a closed contour included in $ \re_{>0}$ for $A$ and in $\mathcal{W}$ for $B$, we can reverse the order of the two integrations by applying Fubini's theorem, and then obtain zero since the quantities inside the integral on $[0, \infty)$ are holomorphic functions in $z$. Hence, the integrals of the lemma define holomorphic functions of $z$, which necessarily coincide with $A$ and $B$ by uniqueness of analytic continuation. 
\end{proof}

%Hence, if $\operatorname{Arg}(z) \in (-\pi/4, \pi/4)$, the initial integral form of $A$ gives 
%$$|A(z)| \ll \int_0^{\infty} \min( 1, |tz|^{-1} ) \min(1, t^{-1}) dt,$$

\subsection{Functional equations for $A$ and the lemma of Bettin and Conrey} \label{bettin-conrey}

Following \cite{BC13a,BC13b}, consider, for $\im(z)>0$, 
\bean
E_1(z) & = & 1-4\sum_{n\ge1}d(n)e^{2i\pi nz} \\
\psi(z) & = & E_1(z)-(1/z)\ E_1(-1/z),
\eean
and the analytic continuation of $\psi$ to $\C'$. %(Recall that $d(n)$ is the number of positive divisors of $n$).
The statement of the following lemma is contained within the proof of \cite[Lemma 1]{BC13b}, completed by the identity (\ref{inverseMellinA}). For the sake of completeness, we recall the arguments of the authors.

\begin{lemm}[Bettin and Conrey 2013, reformulated]
For all $z \in \mathbb{C}'$, the analytic continuations of $A$ and $\psi$ satisfy: 
\bean
    A(z) & = & \frac{i\pi}{4} \psi(z) + r(z),
\eean
where
$\d r(z)  = c \left(\frac{1}{z}+1\right) + \frac{1}{2}\left(\frac{1}{z}-1\right) \log(z)$,\ with\  $c=\frac{\log(2\pi)-\gamma}{2}$.
%= \frac{1}{2}\left(\log(2\pi/x)-\gamma +\frac{\log(2\pi x)-\gamma}{x}\right).
\end{lemm}

\begin{proof}
Recall the notation $\int_{(\beta)}=\int_{\beta - i \infty}^{\beta + i \infty}$, and that if $z\in\HH$, then $-iz=e^{-i\pi/2}z\in\re_{>0}$.

In this case, $|\arg(-2\pi i z)|<\pi/2$. Due to the expansion of $\zeta^2$ into its Dirichlet series, and the functional equation $$\zeta(s)=2^s \pi^{s-1} \sin(\pi s/2)\Gamma(1-s)\zeta(1-s),$$ one has, as in \cite{BC13b}, due to our definition $Q(s) = \Gamma\cdot\zeta(s)\ \Gamma\cdot\zeta(1-s)$, 
\begin{eqnarray*}
    E_1(z) & = & 1- \frac{4}{ 2\pi i} \int_{(2)} \Gamma(s)\zeta(s)^2 (-2\pi i z)^{-s} ds \\
        & = & 1- \frac{4}{ 2\pi^2 i} \int_{(2)} Q(s) \sin(\pi s/2) e^{\pi i s/2} z^{-s} ds.
\end{eqnarray*}
Notice the extra factor $1/\pi$ due to our normalization of $Q$, which is not the same in \cite{BC13b}. 
Since $Q$ has exponential decay at $\pm i \infty$, one can move the line of integration from $(2)$ to $(1/2)$, and then
\begin{eqnarray*}
    E_1(z)   & = & 1- \frac{4}{\pi} \nu (z) - \frac{4}{ 2\pi^2 i} \int_{(1/2)} Q(s) \sin(\pi s/2) e^{\pi i s/2} z^{-s} ds,
\end{eqnarray*}
where (the notation $r$ in \cite{BC13b} is $\nu$ here)
\begin{eqnarray*}
    \nu(z)  & = & \underset{s=1}{\rm Res}\  Q(s) \sin(\pi s/2) e^{\pi i s/2} z^{-s}.
\end{eqnarray*}
In the neighborhood of $1$, we write, putting $s=1+\e$,
\begin{eqnarray*} 
Q(s) \sin(\pi s/2) e^{\pi i s/2} z^{-s}   = \pi \frac{ \zeta(s) \zeta (1-s)}{2\cos (\pi s/2)} e^{\pi i s/2} e^{-s \log z}
\\  =  \frac{\pi}{2z} \frac{( \frac{1}{\e} + \gamma + O(\e)) (-\frac{1}{2} + \frac{\e}{2}\log (2 \pi) + O(\e^2)) }{ -\frac{\pi}{2}\e + O(\e^3)} \ i\left(1+\frac{\pi i}{2} \e+O(\e^2)\right) (1 - \e \log z + O(\e^2)). 
\end{eqnarray*} 
Therefore, as claimed in \cite{BC13b}, 
\begin{eqnarray*}
    \nu(z)  & = & \frac{\log(2\pi z)-\gamma-\pi i/2}{2 iz}.
\end{eqnarray*}
Moreover, for $z \in \mathbb{H}$, i.e. $0<\arg z<\pi$,
\begin{eqnarray*}
    -(1/z)\ E_1(-1/z) & = & -1/z + \frac{4}{\pi z} \nu(-1/z) - \frac{4}{ 2\pi^2 i} \int_{(1/2)} Q(s) \sin(\pi s/2) e^{\pi i s/2} (-1/z)^{1-s} ds \\
        & = & -1/z + \frac{4}{\pi z} \nu(-1/z) - \frac{4}{ 2\pi^2 i} \int_{(1/2)} Q(s) \cos(\pi s/2)\ i  e^{\pi i s/2} z^{-s} ds,
\end{eqnarray*}
where one used the change of variable $s\to 1-s$, the equation $Q(1-s)=Q(s)$, that\\ $\arg(-1/z)=\pi -\arg z$, giving $(-1/z)^s=e^{\pi i s}z^{-s}$, and then $e^{\pi i (1-s)/2} e^{\pi i s} = i e^{\pi i s/2}$.
Noticing that
$$\sin(\frac{\pi s}{2}) e^{i\frac{\pi s}{2}}+ i \cos(\frac{\pi s}{2}) e^{i\frac{\pi s}{2}} 
%= \sin(\frac{\pi s}{2})(\cos(\frac{\pi s}{2})+i\sin(\frac{\pi s}{2})) + i\cos(\frac{\pi s}{2})(\cos(\frac{\pi s}{2})+i\sin(\frac{\pi s}{2})) 
= i,$$
one then obtains
\begin{eqnarray*}
    \psi(z)  =  E_1(z) -(1/z) E_1(-1/z) = 1 - \frac{4}{\pi} \nu(z) - \frac{1}{z} + \frac{4}{ \pi z} \nu (-1/z)
    -\frac{4i }{ 2 \pi^2 i} \int_{(1/2)} Q(s) z^{-s} ds.
\end{eqnarray*}
Now, we write
$$\nu(z) = \frac{c}{iz} - \frac{\pi}{4z}+ \frac{\log z}{ 2 iz}, \quad
(-1/z) \nu(-1/z) = \frac{c}{i} - \frac{\pi}{4} + \frac{\log(-1/z)}{2 i}
= \frac{c}{i} - \frac{\pi}{4} + \frac{ i \pi - \log z}{2 i},
$$
and then 
$$ 
- \frac{4}{\pi} \nu(z)  + \frac{4}{ \pi z} \nu (-1/z)
= -\frac{4}{\pi} \left( \frac{c}{i} (1 + 1/z) + \frac{\pi}{4} (1 - 1/z) - \frac{\log z}{2i} (1 - 1/z) \right),  
$$
which implies, by plugging in $\psi$ and multiplying by $i \pi/4$,
$$\frac{i \pi}{4}\psi(z) = -  c (1 + 1/z) + \frac{\log z}{2} (1 - 1/z)   +\frac{1 }{ 2 \pi i} \int_{(1/2)} Q(s) z^{-s} ds.$$
%i.e. $$\frac{i \pi}{4}\psi(z) = - r(z) + \frac{1 }{ 2 \pi i} \int_{(1/2)} Q(s) z^{-s} ds.$$
We conclude, remarking that the inverse Mellin transform of $Q$ in the critical strip is precisely the function $A$ due to (\ref{inverseMellinA}), and extending the equality from $\mathbb{H}$ to $\mathbb{C}'$ by analytic continuation. 
\end{proof}

This way, the expression of $\psi$ in terms of Eisenstein series provides an expression for the analytic continuation of $A$ in $\HH$, and conversely, the initial definition of $A$ as an integral provides a workable expression for the analytic continuation of $\psi$ in $\re_{>0}$, the function $r$ being easy to handle.\\

%\subsection{Functional equations for $B$ and $A$}
Ramanujan's function $B$ is analytic in the strip $\{x+iy, -\pi<y<\pi\}$ and enjoys two simple functional equations:
$B(\bar{z}) = \overline{B(z)}$ and $B(-z)=B(z)$ (since $B$ is real and even for real $z$). For our purpose, we give the counter part of these equations for $A$, and a more elaborated one (iii).\\

\begin{lemm}
The following functional equations of $A$ hold:
\begin{itemize}
    \item[(i)] Inversion $z\mapsto 1/z$. For all $z\in\CC'$,
    \begin{eqnarray} A(1/z) & = & z A(z).
    \end{eqnarray}
    \item[(ii)] Symmetry $z\mapsto \bar z$. For all $z\in\CC'$,
    \begin{eqnarray} 
    A(\bar{z}) & = & \overline{A(z)}.
    \end{eqnarray}
    \item[(iii)] Symmetry $z\mapsto -z$. For all $z\in\HH$,
    \begin{eqnarray}
    A(z) + A(-z) & = & \frac{2\pi i}{z} S_0(-1/z) + \log(2\pi/z) -\gamma + \frac{\pi}{2}i.
    \end{eqnarray}
\end{itemize}
\end{lemm}

\begin{proof}
The functional equations (i) and (ii) follow by analytic continuation when considering $z\in\R_{>0}$: $A$ is then real and we can use the change of variable $x \mapsto x/z$ in the integral defining $A$. 

Proof of (iii) --- Now let $z\in\HH$. Note that $-\bar z, -1/z, -1/(-\bar{z}) \in \HH$, and that, due to (ii):
\begin{eqnarray}
    A(-z) & = & \overline{A(-\bar{z})}. 
\end{eqnarray}
So we want to compute $\overline{A(-\bar{z})}$.
Since $\overline{e(-\bar{z})}=\overline{e^{-2\pi i\bar z}}=e(z)$, we deduce
\begin{eqnarray*}
\overline{E_1(-\bar{z})} & = & E_1(z) \\
\overline{E_1(-1/(-\bar{z}))} & = & E_1(-1/z).
\end{eqnarray*}
Hence, from $\psi(z) = E_1(z)-(1/z)\ E_1(-1/z)$, we get
\bean
\overline{\psi(-\bar{z})} & = & \overline{E_1(-\bar{z})-(1/(-\bar{z}) )E_1(-1/(-\bar{z}))} \\
    & = & E_1(z) + (1/z)\ E_1(-1/z),
\eean
and then
\bean
A(z) + A(-z) & = & \frac{i\pi}{4}\psi(z) +r(z) +\overline{\frac{i\pi}{4}\psi(-\bar{z})} +r(-z) \\
     & = & -\frac{i\pi}{2 z} E_1(-1/z) +r(z)+r(-z).
\eean
Recall $\d r(z) = c\left(\frac{1}{z}+1\right) + \frac{1}{2}\left(\frac{1}{z}-1\right) \log(z)$. So, with $\log(-z)=-i\pi +\log z$, $0<{\rm Arg}(z)<\pi$, 
\bean
%r(z) & = & C \left(\frac{1}{z}+1\right) + \frac{1}{2}\left(\frac{1}{z}-1\right) \log(z) \\
r(-z) & = & c \left(-\frac{1}{z}+1\right) - \frac{1}{2}\left(\frac{1}{z}+1\right) (-i\pi +\log z),
\eean
and then
\bean
r(z) + r(-z) & = & 2c -\log z +\frac{i\pi}{2z} +\frac{i\pi}{2}.
\eean
Therefore
\bean
A(z) + A(-z) & = &  -\frac{i\pi}{2 z} \left(E_1(-1/z)-1\right) + 2c -\log z +\frac{i\pi}{2},
\eean
which gives the desired expression due to $2c=\log(2\pi)-\gamma$ and the definition of $E_1$.
\end{proof}

\medskip

\subsection{Proof of Theorem \ref{main-lemma}}

We will prove the following instance of Theorem \ref{main-lemma} for all $k\ge 2$,
\begin{eqnarray*}
    \int_{-\infty}^{\infty} 
 \left|\zeta\left(\frac{1}{2}+it\right)\right|^{2k} \frac{e^{ k(\pi -\delta) t}}{\cosh(\pi t)^k} \ dt
& = &
2\ \frac{e^{ik\frac{\delta-\pi}{2}}}{\pi^{k-1} }
 \int_{\RR_{>0}^{k-1}} A\left(\frac{- e^{i \delta}}{u_1\cdots u_{k-1}} \right)
\prod_{j=1}^{k-1} A(- u_j e^{i \delta} ) \frac{du_j}{u_j},
\end{eqnarray*}
and then its general formulation with Dirac surface measures.\\

Recall that we have from (\ref{inversefourierB}), for all real $z$,
\begin{eqnarray*}
B(z) & = & \frac{1}{2 \pi} 
\int_{-\infty}^{\infty} e^{i z t}\left|\zeta\left(\frac{1}{2}+it\right)\right|^{2} \frac{\pi dt}{\cosh(\pi t)}.
\end{eqnarray*}
%The right hand side is analytic on $\cal S=\{x+iy,\ x\in\R, -\pi<y<\pi\}$ and provide an analytic continuation of $B$ on $\cal S$.

Let $B^{k\star}$ be the $k$-fold additive convolution of $B$, that is $B^{k\star}=B\star\cdots \star B$ ($k$ times), which is well-defined since $B\in L^1(\R)$. 

\begin{lemm}
We have for all real $z$,
\begin{eqnarray*}
B^{k\star}(z) & = & \frac{1}{2 \pi} 
\int_{-\infty}^{\infty} e^{i z t}\left|\zeta\left(\frac{1}{2}+it\right)\right|^{2k} \frac{\pi^k dt}{\cosh(\pi t)^k}.
\end{eqnarray*}
\end{lemm}
\begin{proof}
Since $B,B^{k\star}\in L^1(\R)$, we can write $\cal F(B^{k\star}) = (\cal F B)^k$. Using a crude bound on $\zeta(1/2+it)$, we have, recalling (\ref{fourierB}),
\bean
\cal F B(t) = \left|\zeta\left(\frac{1}{2}+it\right)\right|^{2} \frac{\pi}{\cosh(\pi t)} =O(e^{-|t|}).
\eean
Therefore $(\cal F B)^k \in L^1(\R)$ and we can apply the inverse Fourier transform: for all real $z$,
\begin{eqnarray*}
B^{k\star}(z) = \cal F^{-1}((\cal F B)^k)(z) = \frac{1}{2 \pi} 
\int_{-\infty}^{\infty} e^{i z t}\left|\zeta\left(\frac{1}{2}+it\right)\right|^{2k} \frac{\pi^k dt}{\cosh(\pi t)^k},
\end{eqnarray*}
as desired.
\end{proof}

The right hand side of the last equation extends to an analytic function in $$\cal W_k=\{x+iy, -k\pi<y<k\pi\}.$$

We now give a multiple integral expression of $B^{k\star}$, which will later be extended to $\cal W_k$.

\begin{lemm}
Let $k\ge2$. For all real $z$, 
\begin{eqnarray} \label{Bk}
B^{k\star}(z) & = & \int_{\mathbb{R}^{k-1}} B(z/k - x_1 - x_2 - \dots - x_{k-1}) \prod_{j=1}^{k-1} B(z/k + x_j)  \prod_{j = 1}^{k-1} dx_j.
\end{eqnarray}
\end{lemm}

\begin{proof}
Recalling the definition
\begin{eqnarray*}
    B^{(k+1)\star}(z) & = & \int_{\mathbb{R}}   B(z-y_k) B^{k\star}(y_k) dy_k,
\end{eqnarray*}
we obtain, by an immediate induction on $k$, starting from $k=2$ with $\prod_{\varnothing}: =1$, 
%\color{red} Petites typo à corriger :
%\color{black}
\begin{eqnarray*} 
B^{k\star}(z) & = & \int_{\mathbb{R}^{k-1}} B(z - y_{k-1}) \left( \prod_{j = 1}^{k-2} B( y_{j+1}-y_j) \right) B(y_{1}) \prod_{j=1}^{k-1}
dy_j.
\end{eqnarray*}
Using the change of variable $x_{1} = y_{1}$, and for $1 \leq j \leq k-2$, 
$x_{j+1} = y_{j+1}-y_j$,  
we deduce
\begin{eqnarray*} 
B^{k\star}(z) & = & \int_{\mathbb{R}^{k-1}} B(z - x_1 - x_2 - \dots - x_{k-1})  \prod_{j = 1}^{k-1} B(x_j)  \prod_{j=1}^{k-1}
dx_j
\end{eqnarray*}
since $y_{k-1} = x_1 + x_2 + \dots + x_{k-1}$. We deduce the lemma by shifting $x_j$ by $z/k$ for $1 \leq j \leq k-1$.
\end{proof}

\begin{lemm} \label{boundB}
Set $\cal K_0=\{x+iy,\ |x|\leq x_0, |y|\leq y_0\}$ with $x_0>1$ and $1<y_0<\pi$. Then there exist $K_{1}(y_0), K_{2}(x_0,y_0)>0$ such that for all $w\in\cal K_0$ and real $u$,
\begin{eqnarray*}
|B(w+u)| & \leq & K_1(y_0) \\
(1+u^2)|B(w+u)| & \leq & K_2(x_0,y_0).
\end{eqnarray*}
Moreover, the right-hand side of (\ref{Bk}) defines an analytic function on $\cal W_k$. 
\end{lemm}

\begin{proof}
Let $w=x+iy\in\cal K_0$ and $u\in\R$. Then
\bean
B(w+u) & = & \frac{1}{2} 
\int_{-\infty}^{\infty} e^{i u t} e^{ixt-yt}\left|\zeta\left(\frac{1}{2}+it\right)\right|^{2} \frac{  dt}{\cosh(\pi t)}.
\eean
We first deduce that 
\begin{eqnarray} \label{K1}
|B(w+u)| \leq \frac{1}{2} 
\int_{-\infty}^{\infty} e^{y_0 |t|}\left|\zeta\left(\frac{1}{2}+it\right)\right|^{2} \frac{dt}{\cosh(\pi t)} =: K_1(y_0).
\end{eqnarray}
Set $z(t)=\zeta\left(\frac{1}{2}+it\right)\zeta\left(\frac{1}{2}-it\right)$ and  $f(t,x,y)=\frac{1}{2}e^{(ix-y)t} z(t)/\cosh(\pi t)$. 
Differentiating $f$ twice with respect to $t$, and using the fact that $\zeta$, $\zeta'$ and $\zeta''$ have at most polynomial growth on the critical line, one deduces that 
$$|f(t,x,y)| + |\partial_t  f(t, x, y)| + |\partial_t^2  f(t, x, y)| 
\ll e^{y_0 |t|} (x_0^2 + y_0^2) |t|^{O(1)} e^{- \pi |t|},$$
which is integrable on $\mathbb{R}$ and tends to zero at infinity.  

Integrating by part two times, we deduce 
\bean
B(w+u) & = & -\frac{1}{u^2} 
\int_{-\infty}^{\infty} e^{i u t} \partial^2_t f(t,x,y) dt.
\eean
Therefore, there exists $K_2(x_0,y_0)>0$ such that
\bean
(1+u^2)|B(w+u)| & \leq & K_2(x_0,y_0),
\eean
as desired.

Hence, we have for all $z\in k \mathcal{K}_0$ and all real $x_j$, 
\begin{eqnarray*}
\left| B(z/k - x_1 - x_2 - \dots - x_{k-1}) \prod_{j=1}^{k-1} B(z/k + x_j) \right| & \leq & \frac{K_1(y_0)K_2(x_0,y_0)^{k-1}}{(1+x_1^2)\cdots (1+x_{k-1}^2)},
\end{eqnarray*}
which is integrable with respect to the Lebesgue measure on $\R^{k-1}$.
Since any contour in $\mathcal{W}_k$ is included in a compact set of the form $k \mathcal{K}_0$ for suitable values of $x_0 > 1$ and $y_0 \in (1, \pi)$, one can apply Fubini's theorem to any integral 
of $B^{k \star}$ along a closed contour in $\mathcal{W}_k$, and deduce that such contour integral is equal to zero. Hence, 
 the right hand side of (\ref{Bk}) defines a holomorphic function on $\cal W_k$.
\end{proof}
From now, we keep the notation $B^{k\star}$ for its analytic continuation. \\

%The last identity, with a Dirac measure on the surface $x_1+\cdots+x_k=0$, can be obtained by induction from the following lemma (where we explained in more details the Dirac notation):

We can now go back to $A$ through (\ref{Bk}).
Changing $x_j\to u_j=e^{x_j}$, $1\leq j\leq k-1$, we obtain
\begin{eqnarray*}
B^{k\star}(z)   & = & e^{z/2} \int_{\mathbb{R}_+^{k-1}} A((u_1\cdots u_{k-1})^{-1}e^{z/k}) \prod_{j=1}^{k-1} A(u_je^{z/k})  \prod_{j = 1}^{k-1} \frac{du_j}{u_j}.
\end{eqnarray*}
Take $\delta\in(0,\pi)$ and $z=-ik(\pi-\delta)\in \cal W_k$, so 
\bean
e^{i z t} & = & e^{k(\pi-\delta)t} \\
e^{z/2} & = & e^{ik\frac{\delta-\pi}{2}} \\
e^{z/k} & = & - e^{i\delta}.
\eean
Therefore, for all $\delta\in(0,\pi)$,
\begin{equation*}
    \int_{-\infty}^{\infty} 
 \left|\zeta\left(\frac{1}{2}+it\right)\right|^{2k} \frac{e^{ k(\pi -\delta) t}}{\cosh(\pi t)^k} dt
\ =\ 
2\ \frac{e^{ik\frac{\delta-\pi}{2}}}{\pi^{k-1} }
 \int_{\RR_{>0}^{k-1}} A\left(\frac{- e^{i \delta}}{u_1\cdots u_{k-1}} \right)
\prod_{j=1}^{k-1} A(- u_j e^{i \delta} ) \frac{du_j}{u_j},
\end{equation*}
as claimed.\\

We finally give an expression of the $k$-fold convolution in terms of a Dirac surface measure. This is convenient to write a general formula including the case $k=1$. This will also prove fruitful when studying the case $k=3$.
\begin{lemm}
    We have for all $z\in\cal W_k$,
\begin{eqnarray}
    B^{k\star}(z) & = & \int_{\mathbb{R}^k}
\prod_{j=1}^k B(z/k + x_j)
\delta_0 \left( \sum_{j=1}^k x_j \right) \prod_{j = 1}^k dx_j,  \label{Bdistrib}
\end{eqnarray}
and, for $\delta\in(0,\pi)$,
\begin{eqnarray}
    B^{k\star}(-ik(\pi-\delta))
    & = & e^{ik\frac{\delta-\pi}{2}}
 \int_{\RR_{>0}^{k}}  
\prod_{j=1}^{k} A(- u_j e^{i \delta})\ \delta_1\left(\prod_{j = 1}^{k} u_j\right) \prod_{j = 1}^k du_j,
\end{eqnarray}
where $\delta_a(f)dv_1\cdots dv_k$ is the surface measure on the smooth manifold defined by $f(v_1,\ldots,v_k)=0$, which seen as a distribution is $\d \lim_{\e \rightarrow 0^+}\ \varepsilon^{-1} \mathds{1}_{f(v_1,\ldots,v_k) \in [a, a+\varepsilon]}$.
\end{lemm}

\begin{proof}
The first case $k=1$ is only the definition of the Dirac mass at $0$:
\begin{eqnarray*}
    B(z) & = & \int_{\R} B(z+x) \delta_0(x) dx.
\end{eqnarray*}
We now only prove the case $k=2$ for $B$:
\bean
    B\star B(z) & = & \int_{\R^2} B(z/2+x_1)B(z/2+x_2) \delta_0(x_1+x_2) dx_1 dx_2,
    \eean
    the cases $k \geq 3$ being similar with heavier notation. 
Let us compute, with the change of variable $x_2=-x_1+\e x_2'$ for fixed $x_1$,
\bean
\int_{\R^2} B(\frac{z}{2}+x_1)B(\frac{z}{2}+x_2) \frac{\mathds{1}_{x_1+x_2 \in [0, \e]}}{\e}dx_1dx_2 & = & \int_\R B(\frac{z}{2}+x_1)\int_{-x_1}^{-x_1+\e} B(\frac{z}{2}+x_2) \frac{dx_2}{\e}dx_1 \\
    & = & \int_\R B(\frac{z}{2}+x_1)\int_{0}^{1} B(\frac{z}{2}-x_1+\e x_2') dx_2' dx_1.
\eean
The latter expression tends to $B\star B(z)$ as $\e\to0$ by dominated convergence, which can be applied because of the estimates in Lemma \ref{boundB}. 

Finally $\e^{-1}\mathds{1}_{x_1+x_2 \in [0, \e]}$ converges in the sense of distribution to $\delta_0(x_1+x_2)$, which is classically interpreted here as a weighted surface measure on the hyper-surface $x_1+x_2=0$, and denoted by $\delta_0(x_1+x_2) dx_1 dx_2$ (another notation is $d\sigma_f$ with $f(x_1,x_2)=x_1+x_2$ here). This convergence in distribution implies that the limit as $\e\to0$ of the left-hand side is equal to the integral of $B(z/2 + x_1) B(z/2 + x_2)$ with respect to the surface measure $\delta_0 (x_1 + x_2)$, which is then equal to the limit $B \star B (z)$ of the right-hand side. 

To obtain the desired expression with $A$, we write as an approximation in the integral (\ref{Bdistrib}) using the change of variable $u_j=e^{x_j}$,
$$
\frac{\mathds{1}_{0\leq x_1+\ldots+x_k \leq \e}}{\e}\prod_{j=1}^k dx_j  \   = \ \frac{\mathds{1}_{1\leq u_1 \cdots u_k \leq e^\e}}{\e}\prod_{j = 1}^{k} \frac{du_j}{u_j} \\
  \   = \ \frac{e^\e-1}{\e}\cdot \frac{\mathds{1}_{1\leq u_1\cdots u_k \leq 1+e^\e-1}}{e^\e-1}\prod_{j = 1}^{k} \frac{du_j}{u_j},
$$
and then take the limit as $\e\to0^+$, to obtain
\begin{eqnarray*}
    B^{k\star}(z)
    & = & e^{ik\frac{\delta-\pi}{2}}
 \int_{\RR_{>0}^{k}}  
\prod_{j=1}^{k} A(- u_j e^{i \delta})\ \delta_1\left(\prod_{j = 1}^{k} u_j\right) \prod_{j = 1}^k du_j,
\end{eqnarray*}
since the support of the measure is the hyper-surface for which $u_1\cdots u_k=1$ in the integral. The formulation of Theorem \ref{main-lemma} then follows for all $k\ge1$.
\end{proof}

\medskip

\section{Proof of the formulas for the weighted moment $k=1,2,3$}
\label{proofweightedmoment}
We now set
\begin{eqnarray*}
M_{2k}(\delta) & = & \int_{-\infty}^{\infty} 
 \left|\zeta\left(\frac{1}{2}+it\right)\right|^{2k} \frac{e^{ k(\pi -\delta) t}}{\cosh(\pi t)^k} \ dt,
\end{eqnarray*}
and always consider $\delta\in(0,\pi/2)$.

\subsection{The cases $k=1$ and $k=2$}

Let us prove Theorem \ref{moment4} with our method.\\ 

Recall that the quantity $M_2(\delta)$ is the value, for $z = -i(\pi - \delta)$, of the expression 
$B(z)= e^{z/2}A(e^z)$, which is  analytic on $\cal W_1=\{x+iy, -\pi<y<\pi\}$. %Note that $\exp(\cal S)=\C'$ and $A$ is analytic on $\C'$. 
This is nothing but the case $k=1$ of Theorem \ref{main-lemma}:
\begin{eqnarray*}
M_{2}(\delta)    & = & -2i\ e^{i\delta/2} A(-e^{i\delta}) \\
     & = & 2i\ e^{-i\delta/2} A(-e^{-i\delta}),
\end{eqnarray*}
the last equality being due to the fact that $M_2(\delta)$ is real. 
Now from (iii), since $-e^{-i\delta}\in\HH$,
\begin{eqnarray*}
A(-e^{-i\delta}) & = & -A(e^{-i\delta}) + \frac{2\pi i}{-e^{-i\delta}} S_0\left(-\frac{1}{-e^{-i\delta}}\right) + \log(-2\pi/e^{-i\delta}) -\gamma + \frac{i\pi}{2}.
%    & = & -A(e^{i\delta}) + 2\pi i\ e^{-i\delta} S_0(-e^{-i\delta}) + \log(2\pi) -\gamma + \frac{i\pi}{2} - i\delta.
\end{eqnarray*}
Therefore
\begin{eqnarray*}
M_{2}(\delta)     & = & 4\pi \ e^{i\delta/2} \ S_0(e^{i\delta})  + \ 2i\ e^{-i\delta/2} \left( \log(2\pi)-\gamma -\frac{i\pi}{2} - A(e^{-i\delta}) + i\delta \right),
\end{eqnarray*}
as desired.\\

Let us know consider the case of $M_4$. From Theorem \ref{main-lemma} with $k=2$, we have
\begin{eqnarray*}
M_4(\delta) & = & \frac{2}{\pi} e^{i(\delta-\pi)}
 \int_0^\infty A\left(\frac{- e^{i \delta}}{u} \right)
 A(- u e^{i \delta} ) \frac{du}{u}.
\end{eqnarray*}
Since, changing $u\to1/u$,
\begin{eqnarray*}
    \int_1^\infty A\left(\frac{- e^{i \delta}}{u} \right)
 A(- u e^{i \delta} ) \frac{du}{u} & = & \int_0^1 A(- u e^{i \delta} ) A\left(\frac{- e^{i \delta}}{u} \right) \frac{du}{u},
\end{eqnarray*}
we get, using $A(- e^{i \delta}/u ) = -ue^{-i \delta} A(- u e^{-i \delta} )$, 
\begin{eqnarray*}
M_4(\delta) & = & -\frac{4}{\pi} e^{i\delta}
 \int_0^1 A\left(\frac{- e^{i \delta}}{u} \right)
 A(- u e^{i \delta} ) \frac{du}{u} \\
    & = & \frac{4}{\pi} 
 \int_0^1 A(- u e^{-i \delta} )
 A(- u e^{i \delta} ) du \\
    & = & \frac{4}{\pi} \int_0^1  \left|A(- u e^{i \delta} )\right|^2 du.
\end{eqnarray*}
Now, from the functional equation (iii),
\begin{eqnarray*}
A(-u e^{i\delta}) = -A(u e^{i\delta}) + \frac{2\pi i}{u e^{i\delta}} S_0(-1/(u e^{i\delta})) + \log(2\pi/(u e^{i\delta})) -\gamma + \frac{i\pi}{2} = S(u) + R(u),
\end{eqnarray*}
where
\begin{eqnarray} 
S(u) & = & 2\pi i\ e^{-i\delta} S_0(-e^{-i\delta}/u)/u \label{defS}\\
R(u) & = & -A(ue^{i\delta}) -\log(u) + \log(2\pi) -\gamma + \frac{i\pi}{2} - i\delta.  \label{defR}
\end{eqnarray}
Therefore 
\begin{eqnarray*}
\frac{\pi}{4} M_4(\delta) & = & \int_0^1 |S(u)|^2 du +  2\ \re\int_0^1 \b S(u) R(u) du + \int_0^1 |R(u)|^2 du.
\end{eqnarray*}
From (\ref{defS}), we have
\begin{eqnarray*}
\int_0^1 |S(u)|^2 du  = 4\pi^2 \int_0^1 |S_0(-e^{-i\delta}/u)|^2 \frac{du}{u^2} 
= 4\pi^2 \int_1^\infty \left|S_0(-e^{-i\delta}u)\right|^2 du,
\end{eqnarray*}
and then
\begin{eqnarray*}
M_4(\delta) & = & 16\pi \int_1^\infty \left|S_0(-e^{-i\delta}u)\right|^2 du + \wdt{R_1}(\delta) + \wdt{R_2}(\delta),
\end{eqnarray*}
where
\begin{eqnarray*}
    \wdt{R_1}(\delta) & = & \frac{8}{\pi}\ \re\int_0^1 \b S(u) R(u) du \\
    \wdt{R_2}(\delta) & = & \frac{4}{\pi} \int_0^1 |R(u)|^2 du.
\end{eqnarray*}

Let us now estimate these remainders. To study $\wdt{R_2}(\delta)$, we need to bound $A$.\\ If $\operatorname{Arg}(z) \in (-\pi/4, \pi/4)$, the initial integral form of $A$ gives 
$$|A(z)| \ll \int_0^{\infty} 
\min( 1, |tz|^{-1} ) \min(1, t^{-1}) dt,$$
and for $|z| \leq 1$, 
$$|A(z)| \ll \int_0^{1} dt
+ \int_{1}^{|z|^{-1}} t^{-1} dt + \int_{|z|^{-1}}^{\infty} t^{-2} |z|^{-1} dt \ll 1 + \log (1/|z|).$$
Therefore, from (\ref{defR}), $\wdt{R_2}(\delta) \ll 1$ when $0 < \delta < \pi/4$. Moreover, from Titchmarsh \cite{Tit86} p.167, we have
\begin{eqnarray*}
    \int_0^1 |S(u)|^2 du & \ll & \delta^{-1} \log^4(\delta^{-1}).
\end{eqnarray*}
Thus, by the Cauchy-Schwarz inequality, 
\begin{eqnarray*}
    \wdt{R_1}(\delta) & \ll & \delta^{-1/2} \log^2(\delta^{-1}),
\end{eqnarray*}
and the proof of Theorem \ref{moment4} is complete.

\subsection{The case $k=3$}

From Theorem \ref{main-lemma} with $k=3$, we have, since $(e^{-i\pi/2})^3=(-i)^3=i$,
\begin{eqnarray*}
M_{6}(\delta) & = &  \frac{2i}{\pi^2} e^{i\frac{3}{2}\delta} \int_{\R_{>0}^3} 
 A ( -u e^{i\delta})
 A (-v e^{i\delta}) 
 A (-w e^{i\delta})
 \delta_1(uvw) du dv dw.
\end{eqnarray*} 
Since $uvw = 1$ in the integral, two of the variables are in $I_0=(0,1]$ (resp. $I_1=[1, \infty)$) and the last one is in $[1, \infty)$ (resp. $(0,1]$). So we can integrate on the following two subsets:
\begin{eqnarray*}
    J_1 & = & (I_1\times I_0^2) \bigcup (I_0\times I_1\times I_0) \bigcup (I_0^2\times I_1) \\
    J_2 & = & (I_0\times I_1^2) \bigcup (I_1\times I_0\times I_1) \bigcup (I_1^2\times I_0).
\end{eqnarray*}
Using the two simple functional equations of $A$, we can write
$$
A(- e^{i \delta}/u ) = -ue^{-i \delta} A(- u e^{-i \delta}) = -ue^{-i \delta} \overline{A(- u e^{i\delta})}.
$$
Therefore, changing $u,v,w$ to their inverse in $J_2$ and using $uvw = 1$ in the integral, we have
\begin{eqnarray*}
\pi^2 M_6(\delta) & = & \ 2i e^{i\frac{3}{2}\delta} \left( \int_{J_1} \ +\ \int_{J_2} \right) \\
    & = & \ 2i e^{i\frac{3}{2}\delta} \int_{J_1} 
 A ( -u e^{i\delta})
 A (-v e^{i\delta}) 
 A (-w e^{i\delta})
 \delta_1(uvw) du dv dw \\
    &  & -\ 2i\ e^{i\frac{3}{2}\delta} e^{-3i\delta} \int_{J_1} 
 \overline{A ( -u e^{i\delta}) A (-v e^{i\delta}) A(-w e^{i\delta})}
 \delta_1(uvw) du dv dw \\
    & = & 4\ \re \left(i e^{i\frac{3}{2}\delta} \int_{J_1} 
 A ( -u e^{i\delta})
 A (-v e^{i\delta}) 
 A (-w e^{i\delta})
 \delta_1(uvw) du dv dw\right) \\
    & = & 12\ \re \left( i e^{i\frac{3}{2}\delta} \wdt{M_6}(\delta) \right),
\end{eqnarray*} 
where, considering the three symmetric choices in $J_1$,
\begin{eqnarray*}
\wdt{M_6}(\delta)  & = & \int_0^1\int_0^{1} \int_1^{\infty} 
 A ( -u e^{i\delta}) A (-v e^{i\delta}) A (-w e^{i\delta})
 \delta_1(uvw) du dv dw.
\end{eqnarray*} 
For $u, v$ fixed, we have, in the sense of the distributions,
$$\delta_1(uvw)dw = \underset{\varepsilon \rightarrow 0}{\lim}  \frac{\mathds{1}_{uvw \in [1, 1+\varepsilon]}}{\e} 
= \underset{\varepsilon \rightarrow 0}{\lim} \frac{\mathds{1}_{w \in [\frac{1}{uv}, \frac{1+\varepsilon}{uv}]}}{\e}
=
\underset{\eta \rightarrow 0}{\lim} \frac{\mathds{1}_{w \in [\frac{1}{uv},\frac{1}{uv}+\eta]}}{(uv)\eta}
=
\frac{1}{uv}\  \delta_{\frac{1}{uv}}(w)dw,$$
Therefore
\begin{eqnarray*}
\wdt{M_6}(\delta)  & = & \int_0^1\int_0^{1} \int_1^{\infty} 
 A ( -u e^{i\delta}) A (-v e^{i\delta}) A (-w e^{i\delta})
 \delta_1(uvw)dw\ du dv \\
        & = & \int_0^1\int_0^{1} A ( -u e^{i\delta}) A (-v e^{i\delta}) A (- e^{i\delta}/(uv)) \frac{dudv}{uv}\\
        & = & -e^{-i \delta} \int_0^1\int_0^{1} A ( -u e^{i\delta}) A (-v e^{i\delta}) \overline{A(- uv e^{i \delta})} dudv,
\end{eqnarray*} 
where we again used for the last equality
$$
A(- e^{i \delta}/(uv) ) = -uv\ e^{-i \delta} A(- uv e^{-i \delta}) = -uv\ e^{-i \delta} \overline{A(- uv e^{i\delta})}.
$$
Now, we expand the product
\begin{eqnarray*}
    A ( -u e^{i \delta})
 A (-v e^{i \delta})
 \overline{ A (-(uv) e^{i \delta})}
    & = & \left(S(u) + R(u)\right)(S(v) + R(v))\left(\b S(uv) + \b R(uv)\right) \\
%    & = & (S(u)S(v) + S(u)R(v) + S(v)R(u) + R(u)R(v)) \\
%    &   &  \times (\b S(uv) + \b R(uv)) \\
    & = & S(u)S(v)\b S(uv) \\
    &   & + S(u)S(v)\b R(uv) + S(u)\b S(uv)R(v) + S(v)\b S(uv)R(u) \\
    &   & + \b S(uv)  R(u)R(v)+ S(u)R(v)\b R(uv) + S(v)R(u)\b R(uv) \\
    &   & + R(u)R(v)\b R(uv).
\end{eqnarray*}
We have
\begin{eqnarray*} 
S(u)S(v)\b S(uv)  & = & 2\pi i\ \frac{e^{-i\delta}}{u} S_0(-e^{-i\delta}/u) \cdot 2\pi i\ \frac{e^{-i\delta}}{v} S_0(-e^{-i\delta}/v) \cdot 2\pi (-i)\ \frac{e^{i\delta}}{uv} S_0(e^{i\delta}/uv) \\
    & = & 8\pi^3 i\ \frac{e^{-i\delta}}{u^2v^2} 
        S_0(-e^{-i\delta}/u) S_0(-e^{-i\delta}/v) S_0(e^{i\delta}/uv).
\end{eqnarray*}
Therefore
%The main term in $\wdt{M_6}(\delta) $ is then, using (\ref{defS}),
\begin{eqnarray*}
-e^{-i \delta} \int_0^1\int_0^{1} S(u)S(v)\overline{S}(uv) dudv & = & 
  -  8\pi^3 i e^{-2i\delta} \int_0^1\int_0^{1} S_0(-e^{-i\delta}/u) S_0(-e^{-i\delta}/v) S_0(e^{i\delta}/(uv)) \frac{dudv}{u^2v^2} \\
    & = & - 8\pi^3 i e^{-2i\delta} M,
\end{eqnarray*}
where, changing $u\to 1/u$ and $v\to 1/v$,
\begin{eqnarray*}
M & = & \int_1^\infty\int_1^{\infty} S_0(-e^{-i\delta} u) S_0(-e^{-i\delta} v) S_0(e^{i\delta} uv) dudv .
\end{eqnarray*}
We write
\begin{eqnarray*}
\wdt{M_6}(\delta)    & = & - 8\pi^3 i e^{-2i\delta} M-e^{-i\delta} \rho,
\end{eqnarray*}
where $\rho$ is the double integral $\int_0^1\int_0^{1}$ of the seven remaining terms. Therefore
\begin{eqnarray*}
M_6(\delta) 
    & = & \frac{12}{\pi^2}\ \re \left( i e^{i\frac{3}{2}\delta} \wdt{M_6}(\delta) \right) \\
    & = & \frac{12}{\pi^2}\ \re \left( i e^{i\frac{3}{2}\delta} \left(- 8\pi^3 i e^{-2i\delta} M -e^{-i\delta} \rho\right)\right) \\
    & = & 96\pi\ \re \left(e^{-i\frac{\delta}{2}} M \right) -\frac{12}{\pi^2}\ \re \left(i e^{i\frac{\delta}{2}} \rho \right),
\end{eqnarray*} 
which gives the desired expression noting that $\re \left(e^{-i\frac{\delta}{2}} M \right)=\re \left(e^{i\frac{\delta}{2}} \b M \right)$.\\

We now want to estimate the various remainders $R_1,\ldots,R_5$ as defined in the theorem.\\ Let $A_1,A_2,A_3$ be square integrable functions such that $\int_0^u |A_j(v)|^2dv=O(u^\alpha)$, $\alpha>0$, as $u\to 0^+$. (The $A_j$'s will be either $R$ or $S$.) 
The quantity to be bounded is 
\begin{eqnarray*}
    Q^2_{A_1A_2A_3} := \left| \int_0^1\int_0^1 A_1(u)A_2(v)\b A_3(uv) dudv \right|^2 \leq \left(\int_0^1 |A_1(u) | \int_0^1  |A_2(v) \b A_3(uv)| dv\ du\right)^2.
\end{eqnarray*}
We have, using the Cauchy-Schwarz inequality two times,
\begin{eqnarray*}
    Q^2_{A_1A_2A_3}  & \leq & \int_0^1 |A_1(u)|^2 du \int_0^1 \left|\int_0^1  |A_2(v) A_3(uv)| dv\right|^2 du \\
        & \leq & \int_0^1 |A_1(u)|^2 du \int_0^1 |A_2(u)|^2 du  \int_0^1 \int_0^1 |A_3(uv)|^2 dvdu.
        %& \leq & \int_0^1 |A_1(u)|^2 du \int_0^1 |A_2(u)|^2 du  \int_0^1 |A_3(u)|^2 |\log(u)| du.
\end{eqnarray*}
By a change of variable and an integration by parts, we can write
\begin{eqnarray*}
\int_0^1 \int_0^1 |A_3(uv)|^2 dvdu = \int_0^1 \frac{1}{u}\int_0^u  |A_3(v)|^2 dv\ du = - \int_0^1 |A_3(u)|^2 \log(u) du,
\end{eqnarray*}
since $\log(x) \int_0^x |A_3(v)|^2 dv\to 0$ as $x\to0^+$ due to the assumption on $A_3$.
Hence,
\begin{eqnarray*}
Q^2_{A_1A_2A_3}  & \leq & \int_0^1 |A_1(u)|^2 du \int_0^1 |A_2(u)|^2 du  \int_0^1 |A_3(u)|^2 |\log(u)| du.
\end{eqnarray*}
So we need to consider
\begin{eqnarray*}
    Q_S^2 =  \int_0^1 |S(u)|^2 du, & \quad &
    Q_{Sl}^2 = \int_0^1 |S(u)|^2 |\log u| du, \\
    Q_{R}^2 =  \int_0^1 |R(u)|^2 du,  & \quad &
    Q_{Rl}^2 = \int_0^1 |R(u)|^2 |\log u| du.
\end{eqnarray*}

We already know from the previous proof that $Q_S^2 \ll \delta^{-1} \log^4(\delta^{-1})$ and $Q_R^2 \ll 1$.\\ 

Let us now study $Q_{Sl}^2$. For $\delta > 0$ sufficiently small and $u \leq \delta^{2}$, we get
$$
|S(u)| \ll \frac{1}{u} \sum_{n \geq 1} n\ e^{- 6 n \delta /u }
\ll  \frac{1}{u} \sum_{n \geq 1} e^{n \delta /u} e^{- 6 n \delta /u }
\ll  \frac{1}{u} \sum_{n \geq 1} e^{ - 5 n \delta /u }
,
$$ 
and then, since $1/\delta \leq \delta/u$,
$$
|S(u)| \ll \delta^{-1} (\delta/u) e^{-5 \delta/u} \ll \delta^{-1} e^{-\delta/u} \ll \delta^{-1} e^{-\delta^{-1}}.
$$
Thus, noticing that $\int_0^{\delta^2} |\log u| du\ll 1$ and $|\log v|=\log (1/v)\leq \log(\delta^{-2})$ if $\delta^2\leq v\leq 1$,
\begin{eqnarray*}
Q_{Sl}^2 & = & \int_0^{\delta^2} |S(u)|^2 |\log u| du +\int_{\delta^2}^1 |S(v)|^2 |\log v| dv \\
         & \ll & e^{-\delta^{-1}/2} + \log(\delta^{-1}) \int_0^1 |S(v)|^2 dv \\
         & \ll & \delta^{-1}\log^5(\delta^{-1}).
\end{eqnarray*}

Since $R(u)\ll 1+\log(1/u)$, we also have $Q_{Rl}^2 \ll 1$. \\

Finally, it remains to replace the $A_j$ by $R$ or $S$ accordingly to obtain the domination of the various remainders. For instance,
\begin{eqnarray*}
R_1(\delta) \ll Q_{RSS} \ll Q_{R} Q_{S} Q_{Sl} 
         \ll  \delta^{-1}\log^{9/2}(\delta^{-1}).
\end{eqnarray*}
This completes the proof of Theorem \ref{main-theorem}.

\begin{rema} \rm
    Starting form our new formula, the previous method of estimation, which is crude since the Cauchy-Schwarz inequality kills the interaction between the three Eisenstein series, is still sufficiently good to obtain that the main term satisfies: 
    \begin{eqnarray*}
M \ll Q_{SSS} \ll Q_{S} Q_{S} Q_{Sl} 
         \ll  \delta^{-3/2}\log^{13/2}(\delta^{-1}).
\end{eqnarray*}
Notice that the bound $\delta^{-k/2-\e}$ for the $2k$th moment is obtained for all $k\geq 3$ by a more elaborated method in \cite{Tit86} p.173.  
\end{rema}

\bigskip

\section*{Acknowledgement} 
%the unknown referee for his/her very careful reading of their manuscript, corrections and suggestions that improved its presentation. 
The first author is very grateful to Francois Alouges, Michel Balazard, Karim Belabas, Christophe Delaunay, Alessandro Fazzari, Erwan Hillion, and Olivier Ramar\'e for interesting discussions and references. The authors also thank Francois Alouges for his carefull reading and corrections.

Part of this work was done while the first author was visiting IRL Centre de Recherche Math\'ematique, CNRS, Universit\'e de Montr\'eal, supported by Centre National de la Recherche Scientifique (CNRS, France). These institutions are gratefully acknowledged.

\end{document}